\documentclass[12pt, reqno]{article}
\usepackage{amsmath}
\usepackage{amssymb}
\usepackage{amsthm}
\usepackage[dvipsnames,table,xcdraw]{xcolor}
\usepackage[colorlinks=true,
linkcolor=webgreen,
filecolor=webbrown,
citecolor=webgreen]{hyperref}

\definecolor{webgreen}{rgb}{0,.5,0}
\definecolor{webbrown}{rgb}{.6,0,0}

\title{\fontsize{19pt}{19pt}\selectfont On Quasisymmetric Functions with Two Bordering Variables}
\author{
  Andrey Boris Khesin\\
  \fontsize{10pt}{10pt}\selectfont\texttt{khesin@mit.edu}
  \and
  Alexander Lu Zhang\\
  \fontsize{10pt}{10pt}\selectfont\texttt{azhang896@gmail.com}
}
\date{January 7, 2021}
\theoremstyle{plain}
\newtheorem{theorem}{Theorem}[section]

\theoremstyle{definition}
\newtheorem{definition}[theorem]{Definition}
\newtheorem{example}[theorem]{Example}
\theoremstyle{remark}

\newtheorem{question}[theorem]{Question}

\begin{document}
\maketitle

\begin{abstract}
We study a family of formal power series $K_{n, \Lambda}$, parameterized by $n$ and $\Lambda \subseteq [n]$, that largely resemble quasisymmetric functions.  This family was conjectured to have the property that the product $K_{n, \Lambda}K_{m, \Omega}$ of any two formal power series $K_{n, \Lambda}$ and $K_{m, \Omega}$ from the family can be expressed as a linear combination of other formal power series from the same family. In this paper, we show that this is indeed the case and that the span of the $K_{n, \Lambda}$'s forms an algebra.  We also provide techniques for examining similar families of formal power series and a formula for the product $K_{n, \Lambda}K_{m, \Omega}$ when $n=1$. 
\end{abstract}

\section{Introduction}


The concept of quasisymmetric functions was first developed by Stanley \cite{ubc77} in 1972 whose thesis presented the basic theory behind them.  In 1984, Gessel \cite{ubc35} gave the first formal definition of quasisymmetric functions and studied fundamental properties of their Hopf algebra, QSym.  The study of quasisymmetric functions has since grown in popularity due to their numerous connections to other important areas of math such as discrete geometry \cite{ubc35}.

In 2017, Gessel and Zhuang \cite{geszhuang} initiated the study of quasisymmetric functions in relation to proving shuffle-compatibility of permutation statistics by finding certain subalgebras of QSym consisting of quasisymmetric functions related to these statistics.  In 2018, Grinberg \cite{grinberg2018} proceeded to use a similar method to show that the exterior peak statistic set Epk is shuffle-compatible by finding a specific subalgebra of the ring of formal power series.  The particular subalgebra Grinberg \cite{grinberg2018} discovered is the span of a family of functions quasisymmetric in all but two variables and related to the exterior peak set of a permutation. In this paper, we generalize Grinberg's \cite{grinberg2018, grinberg2020} weaker findings by proving that the general family, with sets not limited to exterior peak sets, is also a subalgebra of the ring of formal power series.

Although the formal power series we are studying are not actually quasisymmetric, they can be easily reduced to quasisymmetric functions.  Thus, studying the structure of this subalgebra could allow for further insight into the QSym algebra and shuffle-compatibility of certain permutation statistics, generalizing the findings in \cite{grinberg2018} and \cite{grinberg2020}.  We hope that the result of this paper and the techniques we use can reveal even more about the recently discovered connection between the two aforementioned areas of study.

Section~2 defines some conventions we use throughout the paper and introduces the main theorem that we prove in Section~3.
Section~4 provides a formula for a specific case of the problem.

\section{Preliminary information and main theorem}

This section introduces many of the important definitions we use throughout this paper. We mostly follow the notations established by Grinberg in \cite{grinberg2020}.
This section also states the main theorem of this paper.

Firstly, we let $\mathbb N$ denote the set $\{0, 1, 2, 3, \ldots\}$ of nonnegative integers, which we consider under its natural ordering $0 \prec 1 \prec 2 \prec 3 \prec \cdots$.  We similarly let the set of \textit{natural numbers,} denoted by \text{$\mathbb  N_+$}, be $\{1,2,3, \ldots \}$ and consider it under its natural ordering.  
 
We now extend the natural numbers to include a $0$ term and an $\infty$ term, which we use in the definition of our particular family of formal power series.

\begin{definition}
\label{curlyN}
Let the set of \textit{extended natural numbers}, denoted by $\mathcal{N}$, be the totally ordered set $\{0,1,2,\ldots\} \cup \{\infty\}$, with total order given by $0\prec1\prec2\prec \cdots \prec\infty$.  In essence, $\mathcal{N}$ is $\text{$\mathbb  N_+$}\cup \{0, \infty \}$ with a 0 term
smaller than all terms of $\mathbb  N_+$ and an $\infty$ term larger than all terms of \text{$\mathbb  N_+$}.
\end{definition}

Our family of formal power series is made up of the variables $x_i$, where $i \in \mathcal{N}$.  We now define some notation for a ring of formal power series in these variables.

\begin{definition}\label{fps}
Let $\mathbb{Z}$ denote the ring of integers.  Let $\mathbb{Z} [[ x_0, x_1, x_2, \ldots, x_{\infty}]]$ be the \textit{ring of formal power series} in variables $x_0, x_1, x_2, \ldots, x_{\infty}$. Elements of $\mathbb{Z} [[ x_0, x_1, x_2, \ldots, x_{\infty}]]$ are \textit{formal power series}, potentially infinite $\mathbb Z$-linear combinations of monomials of the form $x_0^{\alpha_0}x_1^{\alpha_1}x_2^{\alpha_2} \cdots x_{\infty}^{\alpha_{\infty}}$ for nonnegative integers $\alpha_0, \alpha_1, \alpha_2, \ldots, \alpha_{\infty}$ (with all but finitely many of $\alpha_0, \alpha_1, \alpha_2, \ldots, \alpha_{\infty}$ being equal to 0); we let the degree of such a monomial be equal to $\alpha_0+ \alpha_1+ \alpha_2+ \cdots+ \alpha_{\infty}$. 
\end{definition}

Throughout this paper, we often write monomials in the form $x_{g_1}x_{g_2}\cdots x_{g_n}$, where $n$ is the degree of the monomial and $g_1, g_2, \ldots, g_n \in \mathcal N$.  We adopt the convention of choosing the order of the subscripts to satisfy the ordering $g_1 \preceq g_2 \preceq \cdots \preceq g_n$.  For example, if we write $x_0^2x_5^3$ as such, then $(g_1, g_2, g_3, g_4, g_5) = (0,0,5,5,5)$.  

To simplify notation in future sections, we provide a distinction between different $x_i$'s.
\begin{definition}\label{diffvars}
We define the \textit{bordering variables} to be $x_0$ and $x_{\infty}$ and the \textit{natural variables} to be $x_i$ for all $i \in \mathbb  N_+$.
\end{definition}

We now define quasisymmetric functions, from whose study our problem arises.

\begin{definition}\label{qsymfunc}
A formal power series $f$ in variables $x_1, x_2, \ldots$ is a \textit{quasisymmetric function} if it has bounded degree and for all $k \in \mathbb  N_+$, any positive integer $k$-tuple of exponents $\alpha_1, \alpha_2, \ldots, \alpha_k$ satisfies the property that the coefficient of the monomial term $x_{i_1}^{\alpha_1}x_{i_2}^{\alpha_2}\cdots x_{i_k}^{\alpha_k}$ in $f$ is equal to that of the monomial term  $x_{j_1}^{\alpha_1}x_{j_2}^{\alpha_2}\cdots x_{j_k}^{\alpha_k}$ in $f$ for any strictly increasing sequences of positive integers $i_1<i_2<\cdots<i_k$ and $j_1<j_2<\cdots<j_k$.
\end{definition}

The functions we study are quasisymmetric in the natural variables $x_1, x_2, \ldots$; i.e., the $x_i$ for $i \in \mathbb  N_+$. Setting $x_0=x_{\infty}=0$ thus results in our functions becoming quasisymmetric.

We are now ready to define the family of formal power series we study for the rest of the paper.
Each of these functions has two parameters, a nonnegative integer $n$ and
a set $\Lambda \subseteq [n]$, where $[n] = \{1, 2, \ldots, n\}$.

\begin{definition}\label{KnLam}
Let $n$ be a nonnegative integer.  If $\Lambda$ is any subset of $[n]$, we define a power series $K_{n, \Lambda}$ as follows:

\begin{equation}\label{KnLeq}
    K_{n, \Lambda}=\hspace{-2 em}\sum\limits_{\substack{(g_1,g_2,\ldots,g_n) \in \mathcal{N}^n;\\0 \preceq g_1 \preceq g_2 \preceq \cdots \preceq g_n \preceq \infty;\\ \text{no } i \in \Lambda \text{ satisfies } g_{i-1}=g_i=g_{i+1}\\(\text{where } g_0=0\text{ and } g_{n+1}=\infty)}}\hspace{-2 em} 2^{|\{g_1,g_2,\ldots,g_n\} \cap \mathbb  N_+|}x_{g_1}x_{g_2}\cdots x_{g_n}.
\end{equation}

\end{definition}



In essence, we sum over all combinations with replacement of $n$ nondecreasing elements $g_1, g_2, \ldots, g_n$ from the set $\mathcal{N}$ satisfying the requirement that there does not exist an $i \in \Lambda$ such that $g_{i-1}=g_i=g_{i+1}$, where $g_0=0$ and $g_{n+1}=\infty$.
For each of these summands, we multiply the monomial $x_{g_1}x_{g_2}\cdots x_{g_n}$ by 2 to the power of the number of distinct natural variables in the set $\{x_{g_1}, x_{g_2}, \ldots, x_{g_n}\}$.  Note that $K_{n, \Lambda}$ is homogeneous of degree $n$.
This family of formal power series defined by (\ref{KnLeq}) is denoted by $K_{n,\Lambda}^{\mathcal Z}$ in \cite{grinberg2018} and \cite{grinberg2020}, but we omit the $\mathcal Z$ notation for convenience.

\begin{example}
Let $n=2$ and $\Lambda = \{1\}$.  Since $g_0=0$, $g_1$, and $g_2$ cannot all be equal, the only restriction is that $g_1$ and $g_2$ cannot equal 0 simultaneously.  Thus, 
\begin{equation}
    K_{2, \{1\}}=\sum\limits_{i \in \text{$\mathbb  N_+$}}2x_0x_i+\sum\limits_{i \in \text{$\mathbb  N_+$}}2x_ix_{\infty}+\sum\limits_{i \in \text{$\mathbb  N_+$}}2x_i^2+\sum\limits_{i<j \in \text{$\mathbb  N_+$}}4x_ix_j+x_{\infty}^2,
\end{equation}
where the coefficient of each monomial $x_{g_1}x_{g_2}\cdots x_{g_n}$ equals 2 to the power of the number of distinct natural variables in the monomial.
\end{example}

At this point, we are ready to present the main result of this paper, a theorem that resolves a conjecture by Grinberg presented as Question 2.51 in \cite{grinberg2018} and Question 4.7 in \cite{grinberg2020} (as of the time of writing).

\begin{theorem}[Main]\label{main}
The span of our family of power series, $\operatorname{span} \left (K_{n, \Lambda} \right )_{n \in \mathbb  N; \text{ } \Lambda \subseteq [n]}$, is a $\mathbb Z$-subalgebra of $\mathbb{Z} [[ x_0, x_1, x_2, \ldots, x_{\infty}]]$.  Equivalently, for any $\Lambda \subseteq [n]$ and any $\Omega \subseteq [m]$, the product $K_{n, \Lambda}K_{m, \Omega}$ is a $\mathbb{Z}$-linear combination from the set $\{K_{n+m,\Xi}\}_{\Xi \subseteq [n+m]}$.
\end{theorem}
Weaker versions of this theorem are proven in \cite{grinberg2018} for sets restricted to exterior peak sets of permutations and \cite{grinberg2020} for $x_0=x_{\infty}=0$.  

We will show later that the base power of 2 in the definition of $K_{n, \Lambda}$ is actually necessary for Theorem \ref{main} to hold true; that is, that the span of the $K_{n, \Lambda}$'s would not form an algebra if the $2^{|\{g_1,g_2,\ldots,g_n\} \cap \mathbb  N_+|}$ coefficient in $K_{n, \Lambda}$ were instead replaced with $q^{|\{g_1,g_2,\ldots,g_n\} \cap \mathbb  N_+|}$ for any other nonzero real number $q$.  

For our proof of Theorem \ref{main}, we define a new family of formal power series $L_{n, \Lambda}$, first introduced by Grinberg in \cite{grinberg2020}, whose definition is very similar to that of $K_{n, \Lambda}$.

\begin{definition}\label{LnLam}
Let $n$ be a nonnegative integer.  If $\Lambda$ is any subset of $[n]$, we define a power series $L_{n, \Lambda}$ as follows:

\begin{equation}
    L_{n, \Lambda}=\hspace{-2 em}\sum\limits_{\substack{(g_1,g_2,\ldots,g_n) \in \mathcal{N}^n;\\0 \preceq g_1 \preceq g_2 \preceq \cdots \preceq g_n \preceq \infty;\\ \textbf{each } i \in \Lambda \text{ satisfies } g_{i-1}=g_i=g_{i+1}\\(\text{where } g_0=0\text{ and } g_{n+1}=\infty)}}\hspace{-2 em} 2^{|\{g_1,g_2,\ldots,g_n\} \cap \mathbb  N_+|}x_{g_1}x_{g_2}\cdots x_{g_n}.
\end{equation}
\end{definition}

Note the distinction between the definitions of $K_{n, \Lambda}$ and $L_{n, \Lambda}$: the bold text below the summation symbol reads ``each'' for $L_{n, \Lambda}$ instead of ``no'' as for $K_{n, \Lambda}$.  

\begin{example}
Let $n=3$ and $\Lambda = \{1\}$.  In this case, $g_0=0$, $g_1$, and $g_2$ must all be equal, so $g_1=g_2=0$,
and there are no restrictions on $g_3$.  Thus, 
\begin{equation}
    L_{3, \{1 \}} = x_0^2\left( x_0+\sum\limits_{i \in \mathbb  N_+} 2x_i + x_{\infty} \right) = x_0^3+\sum\limits_{i \in \mathbb  N_+} 2x_0^2x_i+x_0^2x_{\infty}.
\end{equation}
\end{example}

Two other examples of $L_{n, \Lambda}$'s are $L_{n, [n]}$, which equals 1 for $n=0$ and 0 for all $n>0$ (since we cannot have $0 = \infty$), and $L_{n, \{ \}}$, which equals $K_{n, \{ \}}$ (since there are no restrictions on the variables of either series).

We have now presented our main theorem and defined a new family of formal power series, $L_{n, \Lambda}$.  There exist elegant relations between this family and the original family, $K_{n, \Lambda}$, that we can now use in our proof of the main theorem.

\section{Proof of main theorem}

This section is devoted to the proof of our main result, Theorem \ref{main}, along with a brief proof that the specific base of 2 in the definition of the $K_{n, \Lambda}$'s is necessary for their span to form an algebra. There exists a mutual relation between $K_{n, \Lambda}$ and $L_{n, \Lambda}$ that allows us to reformulate Theorem \ref{main} into a different theorem involving $L_{n, \Lambda}$.

Note that $K_{n, \Lambda}$ contains all monomial terms such that $g_{i-1}=g_i=g_{i+1}$ is false for all $i \in \Lambda$ and $L_{n, \Lambda}$ contains all terms such that $g_{i-1}=g_i=g_{i+1}$ is true for all $i \in \Lambda$. Thus, the inclusion-exclusion principle gives us the following relations between $K_{n, \Lambda}$ and $L_{n, \Lambda}$:
\begin{equation}\label{pie}
    K_{n, \Lambda} = \sum\limits_{\Omega \subseteq \Lambda}(-1)^{|\Omega|}L_{n, \Omega}
\end{equation}
and 
\begin{equation}\label{pie2}
    L_{n, \Lambda} = \sum\limits_{\Omega \subseteq \Lambda}(-1)^{|\Omega|}K_{n, \Omega}.
\end{equation}

Note that we have $\operatorname{span} \left (K_{n, \Lambda} \right )_{n \in \mathbb  N; \text{ } \Lambda \subseteq [n]} \subseteq \operatorname{span} \left (L_{n, \Lambda} \right )_{n \in \mathbb  N; \text{ } \Lambda \subseteq [n]}$ because using (\ref{pie}),  any $\mathbb{Z}$-linear combination from the set $\{K_{n,\Lambda}\}_{\Lambda \subseteq [n]}$ 
can be written in terms of $L_{n, \Lambda}$'s; (\ref{pie2}) analogously implies $\operatorname{span} \left (L_{n, \Lambda} \right )_{n \in \mathbb  N; \text{ } \Lambda \subseteq [n]} \subseteq \operatorname{span} \left (K_{n, \Lambda} \right )_{n \in \mathbb  N; \text{ } \Lambda \subseteq [n]}$.  
This implies $\operatorname{span} \left (K_{n, \Lambda} \right )_{n \in \mathbb  N; \text{ } \Lambda \subseteq [n]} = \operatorname{span} \left (L_{n, \Lambda} \right )_{n \in \mathbb  N; \text{ } \Lambda \subseteq [n]}$.  Therefore, proving Theorem \ref{main} is equivalent to proving the following theorem:

\begin{theorem}[Alternate Main]\label{mainl}
The span of the $L_{n, \Lambda}$'s, $\operatorname{span} \left (L_{n, \Lambda} \right )_{n \in \mathbb  N; \text{ } \Lambda \subseteq [n]}$, is a $\mathbb Z$-subalgebra of $\mathbb{Z} [[ x_0, x_1, x_2, \ldots, x_{\infty}]]$.  Equivalently, for any $\Lambda \subseteq [n]$ and any $\Omega \subseteq [m]$, the product $L_{n, \Lambda}L_{m, \Omega}$ is a $\mathbb{Z}$-linear combination from the set $\{L_{n+m,\Xi}\}_{\Xi \subseteq [n+m]}$.
\end{theorem}

To prove Theorem \ref{mainl}, we start with $L_{n, \Lambda} L_{m, \Omega}$ and ``zero out'' its coefficients by adding and subtracting integer multiples of various $L_{n+m, \Xi}$'s.  We do so by iterating through all $2^{n+m}$ subsets of $[n+m]$ in increasing order by size (the order in which we process terms of the same size does not matter, as we will see in the proof), and we add or subtract an integer multiple of each $L_{n+m, \Xi}$ in order to zero out the coefficients of certain monomial terms.  We then show that it is always possible to zero out all coefficients if this procedure is applied correctly.

The motivation for zeroing out monomial terms in this order comes from the notion of ``generic'' monomial terms in the $L_{n, \Lambda}$'s.  These are monomial terms $\mathbf{x}_g=x_{g_1}x_{g_2}\cdots x_{g_{n}}$ such that for any $i \in \{0 , 1, \ldots, n \}$, $g_i=g_{i+1}$ if and only if $i\in\Lambda$ or $i+1\in\Lambda$.  Informally, these are terms that only have = signs between $g_i$'s in the sequence $g_0\preceq g_1\preceq g_2 \preceq \cdots \preceq g_{n} \preceq g_{n+1}$ when the elements of $\Lambda$ force them to be present (and have $\prec$ signs otherwise).  However, the $L_{n, \Lambda}$'s may also contain monomial terms that aren't ``generic,'' which have = signs between $g_i$'s other than the ones stipulated by $\Lambda$, that might also be present in other $L_{n, \Omega}$'s with $\Lambda \subseteq \Omega$.  Thus, we iterate from smaller sets $\Xi$ to larger ones in order to prevent the coefficients of generic terms in the larger sets we zero out from being changed by processing smaller sets after.

However, this zeroing out does not completely solve the problem, since there are monomials that cannot be expressed as a generic term in some $L_{n, \Lambda}$.  The definition of $L_{n, \Lambda}$ requires that for every element $i \in \Lambda$, $g_{i-1}=g_i=g_{i+1}$ in every monomial term in $L_{n, \Lambda}$. For example, no generic term has $g_{i-1} \prec g_i = g_{i+1} \prec g_{i+2}$ for any $i$; note that any monomial with this last relation will contain $x_{g_i}^2$.  We thus provide a definition that allows us to distinguish between monomials that can be expressed as generic terms and those that cannot.

\begin{definition}
Let $\mathbf m$ be a monomial.  We define $\mathbf m$ to be \textit{$L$-special} if:
\begin{itemize}
\item for all $i \in \mathbb  N_+$, the exponent of $x_i$ in $\mathbf m$
is not equal to 2, and
\item for $i \in \{ 0, \infty \}$, the exponent of $x_i$ in $\mathbf
m$ is not equal to 1.
\end{itemize}
\end{definition}

Note in particular that $L$-special monomials cannot have bordering variables with an exponent of 1 because the relations $g_0 = g_1 = g_2$ and $g_{n-1} = g_n = g_{n+1}$ along with $g_0=0$ and $g_{n+1} = \infty$ require $L$-special monomials to have bordering variables with exponents either equal to 0 or at least 2.  

We now define a family of sets of $L$-special monomials.

\begin{definition}
Let $n$ be a nonnegative integer.  If $\Lambda$ is any subset of $[n]$, we define a set  $M_{n, \Lambda}$ as follows:

\begin{equation}M_{n, \Lambda}=\hspace{-2 em}\bigcup_{\substack{(g_1,g_2,\ldots,g_n) \in \mathcal{N}^n;\\ \text{\textbf{each} }g_i \prec g_{i+1} \text{ if } i,i+1 \notin \Lambda \\(\text{where } g_0=0\text{ and } g_{n+1}=\infty)\\ \text{and } g_i=g_{i+1} \text{ otherwise}}} \hspace{-2 em}\{ x_{g_1}x_{g_2}\cdots x_{g_n}\}.\end{equation}
\end{definition}

We can see that for all $i$ in some $\Lambda$, every term $x_{g_1}x_{g_2}\cdots x_{g_n}$ in $M_{n, \Lambda}$ must satisfy $g_{i-1}=g_i=g_{i+1}$ and that $g_i \prec g_{i+1}$ for all consecutive pairs $g_i$ and $g_{i+1}$ for which neither $i$ nor $i+1$ is in $\Lambda$.  This implies that $M_{n, \Lambda}$ consists, informally, of all terms that only have = signs between $g_i$'s if the elements of $\Lambda$ force the = signs to be present; note that these are exactly the ``generic terms'' in $L_{n, \Lambda}$.

Note that every monomial in any $M_{n, \Lambda}$ must be $L$-special because the definition of $M_{n, \Lambda}$ requires that $g_{i-1}=g_i=g_{i+1}$ for all $i \in \Lambda$, so there cannot be any natural variables with exponent 2 or bordering variables with exponent 1.

In addition, every $L$-special monomial can be expressed as an element of some $M_{n, \Lambda}$ by picking specific values of $i$ in $\Lambda$, so for any nonnegative integer $n$ the union of all $M_{n, \Lambda}$'s is exactly the set of all degree $n$ $L$-special monomials.  For example, the $L$-special monomial $x_0^3 x_1^3 x_4 x_6^5 x_{\infty}^2$ is an element of the set $M_{14, \{ 1, 2, 5, 9, 10, 11, 14 \}}$=$M_{14, \{ 1, 2, 5, 9, 11, 14 \}}$.  
Finally note that by Definition \ref{LnLam}, an $L_{n, \Lambda}$ contains some element of an $M_{n, \Omega}$ as a monomial term if and only if $\Lambda \subseteq \Omega $ or $L_{n, \Lambda} = L_{n, \Omega}$. 

We provide one final definition that allows for a convenient notation of the coefficients of various monomials in our process of zeroing them out.

\begin{definition}
Let $f$ be a formal power series in $\mathbb{Z} [[ x_0, x_1, x_2, \ldots, x_{\infty}]]$.  For every monomial $\mathbf m$, we denote the coefficient of $\mathbf m$ in $f$ by $[\mathbf m] (f)$.
\end{definition}
We are now ready to prove our main theorem.  Within this proof, we will additionally work through the steps of expressing the particular product $L_{2, \{1\}} L_{3, \{2\}}$ as a sum of $L_{5, \Xi}$'s in order to demonstrate this process more clearly.

\begin{proof}[Proof of Theorem \ref{mainl}]

Consider any $L$-special monomial $\mathbf{x}_g=x_{g_1}x_{g_2}\cdots x_{g_{n+m}}$ in the expansion of $L_{n, \Lambda}L_{m, \Omega}$.  This $\mathbf x_g$ is in $M_{n+m, \Xi}$ for some $\Xi \subseteq [n+m]$.  Let $T$ be the power series that begins equal to $L_{n, \Lambda}L_{m, \Omega}$ and from which we add or subtract terms to zero out its coefficients. Assume we have zeroed out the coefficients of all monomials in $M_{n+m, \Psi}$ for all sets $\Psi$ processed before $\Xi$ and added or subtracted terms accordingly from $T$; recall that we process sets in increasing order of size, and that we will later show that the order in which we process sets of the same size does not matter.
We show that it is now possible to zero out $[\mathbf{x}_g](T)$ by adding an integer multiple of $L_{n+m, \Xi}$ to $T$ and that doing so does not affect any monomials we have already zeroed out.

In the example of $L_{2, \{1\}} L_{3, \{2\}}$, we would begin with $T=L_{2, \{1\}} L_{3, \{2\}}$ and process sets in increasing order of size; we would begin with the empty set, then process all six sets of size 1, and then continue to process sets up through the set $\{1,2,3,4,5\}$ of size 5.  Each time we process a set $\Xi$, we would add some multiple $cL_{5, \Xi}$, where $c \in \mathbb Z$, of $L_{5, \Xi}$ from $T$ so that $[\mathbf x_g](T-cL_{5, \Xi})=0$ for any $\mathbf x_g \in M_{5, \Xi}$.  Again, we will later observe that the order in which we process sets of the same size does not matter.

Note that $[\mathbf{x}_g](L_{n, \Lambda}L_{m, \Omega})$ must be divisible by $[\mathbf{x}_g](L_{n+m, \Xi}) = 2^{|\{g_1, g_2, \ldots, g_{n+m}\} \cap \mathbb  N_+|}$ because if two terms $x_{i_1}x_{i_2}\cdots x_{i_n}$ in $L_{n, \Lambda}$ and $x_{i_{n+1}}x_{i_{n+2}}\cdots x_{i_{n+m}}$ in $L_{m, \Omega}$ multiply to $\mathbf{x}_g$, then the exponent of 2 in $[\mathbf{x}_g](L_{n, \Lambda}L_{m, \Omega})$, $|\{i_1, i_2, \ldots, i_{n}\} \cap \mathbb  N_+|+|\{i_{n+1}, i_{n+2}, \ldots, i_{n+m}\} \cap \mathbb  N_+|,$ must be greater than or equal to $|\{g_1, g_2, \ldots, g_{n+m}\} \cap \mathbb  N_+|$ since the multisets $\{g_1, g_2, \ldots, g_{n+m}\}$ and $\{i_1, i_2, \ldots, i_{n+m}\}$ are equal, so
$2^{|\{g_1, g_2, \ldots, g_{n+m}\} \cap \mathbb  N_+|}$ divides $2^{|\{i_1, i_2, \ldots, i_{n}\} \cap \mathbb  N_+|+|\{i_{n+1}, i_{n+2}, \ldots, i_{n+m}\} \cap \mathbb  N_+|}$.  Also note that any term $L_{n+m, \Psi}$ we have already added or subtracted from $T$ that might affect our coefficient of $\mathbf{x}_g$ simply changes $\mathbf{x}_g$'s coefficient by $2^{|\{g_1, g_2, \ldots, g_{n+m}\} \cap \mathbb  N_+|}$.  This shows that $2^{|\{g_1, g_2, \ldots, g_{n+m}\} \cap \mathbb  N_+|}$ divides $[\mathbf{x}_g](T)$, so it is possible to zero out $[\mathbf{x}_g](T)$ by adding or subtracting some multiple of $L_{n+m, \Xi}$ to or from $T$. 

Furthermore, adding or subtracting this multiple of $L_{n+m, \Xi}$ does not affect the coefficients of any monomials in some $M_{n+m, \Psi}$ that has already been processed; this is because an $L_{n, \Lambda}$ contains some element of an $M_{n, \Omega}$ as a monomial term if and only if $\Lambda \subseteq \Omega$ or $L_{n, \Lambda} = L_{n,\Omega}$.  In particular, for any monomial in an $M_{n+m, \Psi}$ for an already-processed $\Psi$, it is not possible for $\Xi \subseteq \Psi$ to be true because $|\Psi| \leq |\Xi|$ (since we process in increasing order of size) and $\Psi \neq \Xi$.  Moreover, if $L_{n+m, \Xi} = L_{n+m, \Psi}$, then $M_{n+m, \Xi} = M_{n+m, \Psi}$ as well, meaning the multiple of $L_{n+m, \Xi}$ we add would simply equal 0 and thus not affect the coefficients of any monomial in $M_{n+m, \Psi}$. Therefore, by going through every subset of $[n+m]$ in increasing order of size, we are able to zero out the coefficient of every $L$-special monomial.

To see this process for the particular product $L_{2, \{1\}} L_{3, \{2\}}$, note that $L_{2, \{1\}} = x_0^2$ and $L_{3, \{2\}} = x_0^3+\sum\limits_{i \in \mathbb N_+} 2x_i^3+x_{\infty}^3$, so the product $L_{2, \{1\}} L_{3, \{2\}}$, and our initial value for $T$, expands to $x_0^5+x_0^2\sum\limits_{i \in \mathbb N_+} 2x_i^3+x_0^2x_{\infty}^3$.  None of the monomials in this sum are members of any $M_{5, \Xi}$'s for $\Xi$'s of size 0 or 1; however, note that all monomials of the form $x_0^2x_i^3$ for $i \in \mathbb N_+$ are members of $M_{5, \{1,4\}}$.  Thus, we subtract $L_{5, \{1,4\}} = x_0^2\sum\limits_{i \in \mathbb N_+} 2x_i^3$ from $T$, making $T = x_0^5+x_0^2x_{\infty}^3$.  Then, because $x_0^5 = L_{5, \{1,2,4\}}=L_{5, \{1,3,4\}}=L_{5, \{1,2,3,4\}}$ and $x_0^2x_{\infty}^3 = L_{5, \{1,4,5\}}$, we can subtract $L_{5, \{1,4,5\}}$ and one of the $L_{5, \Xi}$'s equal to $x_0^5$ from $T$ so that $T=0$, as desired; the order in which we process these sets, and in particular the sets of the same size, does not matter.  Note that all monomials in $L_{2, \{1\}} L_{3, \{2\}}$ are $L$-special, making the next part of our proof unnecessary.

It now remains to show that after zeroing out the coefficients of all $L$-special monomials, any monomial that is not $L$-special also has a coefficient of 0 in $T$.
Consider such a monomial $\mathbf{x}_g=x_{g_1}x_{g_2}\cdots x_{g_{n+m}}$.  
Since $\mathbf x_g$ is not $L$-special, at least one of the following must be true:

\begin{itemize}
    \item The exponent of $x_{g_i}$ in $\mathbf x_g$ is 2 for some $g_i \in \mathbb  N_+$.
    Equivalently, $g_{i-1} \prec g_i = g_{i+1} \prec g_{i+2}$ in $\mathbf x_g$ for some $i \in \{1, 2, \ldots, n-1\}$.
    \item The exponent of either $x_0$ or $x_{\infty}$ in $\mathbf x_g$ is 1.
    Equivalently, either $g_0 = g_1 \prec g_2$ or $g_{n-1} \prec g_n = g_{n+1}$ in $\mathbf x_g$.
\end{itemize}

Let the relations above be the \textit{problematic relations}; the presence of any one of these in $\mathbf x_g$ prevents it from being $L$-special by the definition of $L$-special monomials.
Consider one such problematic relation in $\mathbf x_g$ and let it be the \textit{current problematic relation}.
Note that all of the problematic relations contain exactly one = sign.  Consider a different monomial $\mathbf x_h=x_{h_1}x_{h_2} \cdots x_{h_n}$ 
that is defined with the same relations between consecutive pairs $h_i$ and $h_{i+1}$ as $\mathbf x_g$'s $g_i$ and $g_{i+1}$ but with the current problematic relation resolved by replacing its = sign with a $\prec$ sign\footnote{For example, consider $\mathbf x_g = x_0^2x_1^2x_2$, where $(g_1, g_2, g_3, g_4, g_5) = (0,0,1,1,2).$  The problematic relation is $g_2 \prec g_3 = g_4 \prec g_5$, so a possible monomial $\mathbf x_h$ formed by resolving it might be $x_0^2x_1x_2x_3$, since $0 \prec 1 \prec 2 \prec 3$.}.  We define a formal power series $f$ to have the \textit{spreading condition} if $2[\mathbf x_g](f) = [\mathbf x_h](f)$ for all monomials $\mathbf x_g$ that have a problematic relation and all monomials $\mathbf x_h$ obtained from $\mathbf x_g$ by resolving the problematic relation. We claim that both $L_{n, \Lambda} L_{m, \Omega}$ and all $L_{n+m, \Xi}$'s satisfy the spreading condition.

Before proving the above claim, we first show that proving it would be sufficient to prove Theorem \ref{mainl}.  If $N$ is the number of problematic relations in $\mathbf x_g$, then the coefficient of the monomial formed by resolving every problematic relation in $\mathbf x_g$ would be $2^{N}$ times the coefficient of $\mathbf x_g$ in both $L_{n, \Lambda} L_{m, \Omega}$ and all $L_{n+m, \Xi}$'s.  However, the resulting monomial would be an $L$-special monomial because it would not have any problematic relations, so its coefficient would be 0 after the coefficients of all $L$-special monomials are zeroed out.  This would imply that the coefficient of $\mathbf x_g$ is 0 after all zeroing out is complete, and since this would be true for all non-$L$-special monomials $\mathbf x_g$, proving this claim would be sufficient to complete the proof of Theorem \ref{mainl}.

To prove the claim, we first show that any $L_{n+m, \Xi}$ satisfies the spreading condition.  We can see this is true because by the definition of $L_{n, \Lambda}$, the coefficient of any monomial in $L_{n+m, \Xi}$ is 2 to the power of the number of distinct natural variables in $\Xi$.  Changing an = sign to a $\prec$ sign in any problematic relation increases the number of distinct natural variables by 1, thus doubling the coefficient, which implies that $2[\mathbf x_g](L_{n+m, \Xi}) = [\mathbf x_h](L_{n+m, \Xi})$.

It thus remains to prove that $L_{n, \Lambda} L_{m, \Omega}$ satisfies the spreading condition, or that the analogous equality $2[\mathbf x_g](L_{n, \Lambda} L_{m, \Omega}) = [\mathbf x_h](L_{n, \Lambda} L_{m, \Omega})$ is true.  To do so, we consider two separate cases depending on the type of the current problematic relation in $\mathbf x_g$.

Note that the $\mathbf x_g$ term in $L_{n, \Lambda} L_{m, \Omega}$ is formed by the sum of products of two monomials, one from each of $L_{n, \Lambda}$ and $L_{m, \Omega}$.  In each of these products, each of the single variables $x_{g_1}, x_{g_2}, \ldots, x_{g_{n+m}}$ that multiply to $\mathbf x_g$ must come from one of the two factors $L_{n, \Lambda}$ and $L_{m, \Omega}$.   Consider any pair of monomials $\mathbf x_{g_1}=x_{g_{a_1}} x_{g_{a_2}} \cdots x_{g_{a_n}}$ from $L_{n, \Lambda}$ and $\mathbf x_{g_2} = x_{g_{a_{n+1}}} x_{g_{a_{n+2}}} \cdots x_{g_{a_{n+m}}}$ from $L_{m, \Omega}$ whose product equals $\mathbf x_g$, where $\{ a_1, a_2, \ldots, a_{n+m} \} = [n+m]$, $a_1<a_2< \cdots < a_n$, and $a_{n+1}<a_{n+2} < \cdots < a_{n+m}$.  The coefficient of $\mathbf x_{g_1}$ in $L_{n, \Lambda}$ is $2^{|\{ g_{a_1}, g_{a_2}, \ldots, g_{a_n} \} \cap \mathbb  N_+|}$ and that of $\mathbf x_{g_2}$ in $L_{m, \Omega}$ is $2^{|\{ g_{a_{n+1}}, g_{a_{n+2}}, \ldots, g_{a_{n+m}} \} \cap \mathbb  N_+|}$.  Multiplying these two coefficients gives the product $2^{|\{ g_{a_1}, g_{a_2}, \ldots, g_{a_n} \} \cap \mathbb  N_+|+|\{ g_{a_{n+1}}, g_{a_{n+2}}, \ldots, g_{a_{n+m}} \} \cap \mathbb  N_+|}$; summing over all possible $\mathbf x_{g_1}$ and $\mathbf x_{g_2}$ gives the coefficient $[\mathbf x_g](L_{n, \Lambda}L_{m, \Omega})$ of $\mathbf x_g$ in $L_{n, \Lambda}L_{m, \Omega}$. 

Consider the first type of problematic relation, when $g_{i-1} \prec g_i = g_{i+1} \prec g_{i+2}$ in $\mathbf x_g$ for some $i \in \{1, 2, \ldots, n-1\}$. Note that $x_{g_i}$ and $x_{g_{i+1}}$, which are equal in $\mathbf x_g$, must either come from different factors, one in $L_{n, \Lambda}$ and one in $L_{m, \Omega}$, or from the same factor in $L_{n, \Lambda}$ or $L_{m, \Omega}$.  This means there are three possibilities for which factors $x_{g_i}$ and $x_{g_{i+1}}$ come from: both of $x_{g_i}$ and $x_{g_{i+1}}$ come from $\mathbf x_{g_1}$, both of $x_{g_i}$ and $x_{g_{i+1}}$ come from $\mathbf x_{g_2}$, or one of $x_{g_i}$ and $x_{g_{i+1}}$ comes from each of $\mathbf x_{g_1}$ and $\mathbf x_{g_2}$.

If they come from the same factor, consider the pair of monomials $x_{h_{a_1}} x_{h_{a_2}} \cdots x_{h_{a_n}}$ from $L_{n, \Lambda}$ and $x_{h_{a_{n+1}}} x_{h_{a_{n+2}}} \cdots x_{h_{a_{n+m}}}$ from $L_{m, \Omega}$, whose product equals $\mathbf x_h$.  The coefficient of this product in $L_{n, \Lambda} L_{m, \Omega}$ is $2^{|\{ h_{a_1}, h_{a_2}, \ldots, h_{a_n} \} \cap \mathbb  N_+|+|\{ h_{a_{n+1}}, h_{a_{n+2}}, \ldots, h_{a_{n+m}} \} \cap \mathbb  N_+|}$, which is twice that of the product of $\mathbf x_{g_1}$ and $\mathbf x_{g_2}$ because there is one more distinct value of the $h$'s than the $g$'s because the current problematic relation in $\mathbf x_g$ was resolved to form $\mathbf x_h$.

If they come from different factors, let $a_j = i$ and $a_k=i+1$.  We consider two pairs of monomials whose products are $\mathbf x_h$: firstly $x_{h_{a_1}} x_{h_{a_2}} \cdots x_{h_{a_n}}$ from $L_{n, \Lambda}$ and $x_{h_{a_{n+1}}} x_{h_{a_{n+2}}} \cdots x_{h_{a_{n+m}}}$ from $L_{m, \Omega}$, and secondly the same pair but with $a_j$ and $a_k$ swapped.  The coefficients of the products of both of these monomials in $L_{n, \Lambda} L_{m, \Omega}$ are equal to that of the product of $\mathbf x_{g_1}$ and $\mathbf x_{g_2}$ because $x_{g_i}$ and $x_{g_{i+1}}$ are not part of the same factor and hence are counted twice in the exponent.  However, there are two monomials to consider, so the total sum of their coefficients is twice that of the product of $\mathbf x_{g_1}$ and $\mathbf x_{g_2}$.  

Now consider the second type of problematic relation, when either $g_0 = g_1 \prec g_2$ or $g_{n-1} \prec g_n = g_{n+1}$ in $\mathbf x_g$.  In either case, consider the pair of monomials $x_{h_{a_1}} x_{h_{a_2}} \cdots x_{h_{a_n}}$ from $L_{n, \Lambda}$ and $x_{h_{a_{n+1}}} x_{h_{a_{n+2}}} \cdots x_{h_{a_{n+m}}}$ from $L_{m, \Omega}$, whose product equals $\mathbf x_h$.  The coefficient of this product in $L_{n, \Lambda} L_{m, \Omega}$ is $2^{|\{ h_{a_1}, h_{a_2}, \ldots, h_{a_n} \} \cap \mathbb  N_+|+|\{ h_{a_{n+1}}, h_{a_{n+2}}, \ldots, h_{a_{n+m}} \} \cap \mathbb  N_+|}$, which is twice that of the product of $\mathbf x_{g_1}$ and $\mathbf x_{g_2}$ because the only difference between the $x_{g_a}$'s and the $x_{h_a}$'s is that a bordering variable is replaced with a natural variable.  

Since for both types of problematic relations the coefficient of $\mathbf x_h$ is double that of $\mathbf x_g$ in $L_{n, \Lambda} L_{m, \Omega}$, summing over all pairs of monomials multiplying to $\mathbf x_g$ implies that the total coefficient of $\mathbf x_h$ is double that of $\mathbf x_g$, or that $2[\mathbf x_g](L_{n, \Lambda} L_{m, \Omega}) = [\mathbf x_h](L_{n, \Lambda} L_{m, \Omega})$, as desired.  This completes our proof of Theorem \ref{mainl} and, equivalently, Theorem \ref{main}.
\end{proof}

Though the choice of 2 as the base of the exponent in $K_{n, \Lambda}$ might seem somewhat arbitrary, there are actually no other nonzero bases for which the span of the $K_{n, \Lambda}$'s would form an algebra.  We provide a brief proof of this below.

\begin{proof}
Consider an analogous definition of the $K_{n, \Lambda}$'s but with the coefficients' base of 2 replaced by an arbitrary, nonzero constant $q$:

\begin{equation}
     K_{n, \Lambda}=\hspace{-2 em}\sum\limits_{\substack{(g_1,g_2,\ldots,g_n) \in \mathcal{N}^n;\\0 \preceq g_1 \preceq g_2 \preceq \cdots \preceq g_n \preceq \infty;\\ \text{no } i \in \Lambda \text{ satisfies } g_{i-1}=g_i=g_{i+1}\\(\text{where } g_0=0\text{ and } g_{n+1}=\infty)}}\hspace{-2 em} q^{|\{g_1,g_2,\ldots,g_n\} \cap \mathbb{N}|}x_{g_1}x_{g_2}\cdots x_{g_n}.
\end{equation}
Now, consider the product \begin{equation}
    K_{1, \{ \}}^2 = \left (x_0+\sum\limits_{i \in \mathbb N} qx_i + x_{\infty} \right )^2.
\end{equation}
The coefficient of $x_i^2$ for any $i \in \mathbb N$ is $q^2$ and the coefficient of $x_0x_i$ for any $i \in \mathbb N$ is $2q$ in this formal power series.  However, the coefficients of $x_i^2$ and $x_0x_i$ are both $q$ in any $K_{2, \Lambda}$.  This implies that a $\mathbb Z$-linear combination of $K_{2, \Lambda}$'s will have the same coefficient of $x_i^2$ and $x_0x_i$, so for $K_{1, \{ \}}^2$ to be expressible as such a linear combination $q^2$ must equal $2q$, implying $q=2$.  This proves that the base of 2 is necessary for the $K_{n, \Lambda}$'s to satisfy Theorem \ref{main}.
\end{proof}

\section{Formula for  \texorpdfstring{$K_{1, \Lambda}K_{m, \Omega}$}{} in terms of  \texorpdfstring{$K_{m+1, \Xi}$}{}'s}

We now present a direct formula for the product $K_{1, \Lambda}K_{m, \Omega}$ as a sum of $K_{m+1, \Xi}$'s.  To do so, we generalize the notions of peak sets and permutation shuffles provided in \cite{grinberg2018} for use outside of the context of permutations.

We begin by defining a way to write subsets of $[n]$ as strings of length $n$ that help us describe some of our formulas.

\begin{definition}\label{StringSets}Let $n$ be a nonnegative integer and let $(X,Y)$ be an ordered pair of distinct letters.  Given any subset $\Lambda$ of $[n]$, we define the \textit{string form} of $\Lambda$ as an $n$-letter string $s$ consisting of only $X$'s and $Y$'s such that the $i$th character of $s$ is $X$ if $i \in \Lambda$ and is $Y$ otherwise. 
\end{definition}

For example, if $n=5$, the string form of the set $\{1, 4, 5 \}$ is $XYYXX$.  For the rest of this paper, we always choose our ordered pairs of letters to be $(A,B)$ and $(C,D)$ for the sake of simplicity. We do not need to refer to more than two subsets simultaneously.

We now define the notion of a shuffle of two strings, generalizing the notion of shuffles of two permutations provided in \cite{grinberg2018}.

\begin{definition}\label{StringShuffles}
Let $n$ and $m$ be nonnegative integers, $\Lambda \subseteq [n]$, and $\Omega \subseteq [m]$.  If $z_1$ is the string form of $\Lambda$ with letters $(A,B)$ and $z_2$ is the string form of $\Omega$ with letters $(C,D)$, then we define a string $z_3$ of length $n+m$ consisting of the four letters $A,B,C,$ and $D$ to be a \textit{shuffle} of $\Lambda$ and $\Omega$ if $z_1$ and $z_2$ both occur as disjoint but not necessarily contiguous subsequences of $z_3$.
\end{definition}

For example, if $n=2$, $m=3$, $\Lambda = \{1\}$, and $\Omega = \{2,3\}$, then there are ten shuffles of $\Lambda$ and $\Omega$, such as $ABDCC$ and $DCACB$.  
Note that all subsequences of these shuffles consisting of the letters $A$ and $B$ are of the form $AB$ and all subsequences of these shuffles consisting of the letters $C$ and $D$ are of the form $DCC$.

We now introduce a more convenient notation to express the set of all shuffles of two given sets.

\begin{definition}\label{SetofShuffles}
Denote the set of shuffles of two subsets $\Lambda$ and $\Omega$ of $[n]$ and $[m]$ respectively by $S(n, \Lambda, m, \Omega)$.  It is not difficult to see that the size of $S(n, \Lambda, m, \Omega)$ is $\binom{n+m}{n}$.
\end{definition}

We now impose an ordering $A>B>C>D$ on the letters $A,B,C,$ and $D$ of $S(n, \Lambda, m, \Omega)$.  This ordering allows us to identify each shuffle with a subset of $[n+m]$ by generalizing the notion of a peak set presented in \cite{grinberg2018}, which is the crux of the formula for $K_{1, \Lambda}K_{m, \Omega}$.  

\begin{definition}\label{GPS}
Denote by $\text{Gp}(s)$ the \textit{generalized peak set corresponding to }$s$, which for a string $s$ with letters from the set $\{A, B, C, D\}$ is the set of positions in $s$ satisfying the following properties:
\begin{itemize}
    \item The letters at those positions are not the smallest letter under the ordering $A>B>C>D$, namely the letter $D$.
    \item The letters at those positions are greater or equal to all of their neighbors.
\end{itemize}
\end{definition}

As an example, the set corresponding to the string $BCACDD$ is $\{1,3\}$ under the ordering $A>B>C>D$.  Positions 1 and 3 satisfy both properties necessary for inclusion in the generalized peak set corresponding to $BCACDD$, since $A,B \geq C$.  Note that while position 6 satisfies the second property since $D \geq D$, it does not satisfy the first, so position 6 is not included in the set.

We are now ready to present a formula for $K_{1, \{ \}}K_{m, \Omega}$, where $\{ \}$ denotes the empty set.

\begin{theorem}\label{generalizedpeakshuffletheorem}
Let $m$ be a nonnegative integer and $\Omega \subseteq [m]$.  The product $K_{1, \{ \}}K_{m, \Omega}$ is equal to the sum of the terms of the form $K_{m+1, \Lambda}$, where $\Lambda$ is the generalized peak set of each shuffle in $S(1, \{ \}, m, \Omega)$ under the ordering $A>B>C>D$.  This formula is denoted as a sum by 
\begin{equation}\label{formula}
    K_{1, \{ \}}K_{m, \Omega}=\sum\limits_{s \in S(1, \{ \}, m, \Omega)} K_{m+1, \text{Gp}(s)}.
\end{equation}
\end{theorem}

For example, the product $K_{1, \{ \}}K_{5, \{1,2,4 \}}$ can be written as $K_{6,\{2,5\}}+K_{6,\{1,2,4\}}+K_{6,\{1,2,5\}}+K_{6,\{1,3,5\}}+K_{6,\{1,3,5\}}+K_{6,\{1,2,4,6\}}$.  This is consistent with (\ref{formula}), which we can verify by finding the generalized peak set of each shuffle of $S(1, \{ \},5, \{1,2,4\})$ under the ordering $A>B>C>D$.  The proof of this theorem is both highly technical and not very enlightening, so we save it for Appendix \ref{apdx}.

Note that $K_{1, \{ \}} = K_{1, \{ 1\}}$ by Definition \ref{KnLam} since there are no restrictions on $g_1$ for either of the sets $\{ \}$ and $\{ 1\}$. Thus, the above formula holds for $K_{1, \{1\}}$ in place of $K_{1, \{ \}}$ as well, so we have found a formula that works for all $K_{1, \Lambda}K_{m, \Omega}$.

\begin{question}
Can we use the notion of the generalized peak set to find a formula for all products $K_{n, \Lambda}K_{m, \Omega}$? 
\end{question}
Such a formula has already been found for special cases such as when $\Lambda$ and $\Omega$ are exterior peak sets \cite{grinberg2018}.  Furthermore, some testing with SageMath \cite{sage} has hinted at some potentially useful directions of study in this area. For example, the products $K_{2, \{ \} }K_{2, \{ 1 \}}$ and $K_{2, \{ 1 \}}K_{3, \{ 1,2 \}}$ can be written as sums similar to that of (\ref{formula}), and some recursive patterns for products of the form $K_{n, \{ \}}K_{m, \Omega}$ have been identified, suggesting that a general formula may require a complete characterization and solution to similar recurrences.

One large difficulty in finding a general formula for any $K_{n,\Lambda} K_{m, \Omega}$, in comparison with products of the form $K_{n, \{ \}}K_{m, \Omega}$, arises due to the additional parameter $\Lambda$.  This would require a formula that would account for the many varying possibilities for both sets $\Lambda$ and $\Omega$ in contrast to the progress we have made for the fixed $\Lambda = \{ \}$.  In addition, another complication for determining the general product $K_{n,\Lambda} K_{m, \Omega}$ lies in the fact that simply summing $K_{n+m,\Xi}$ terms may not be sufficient to produce a correct equation.  
For example, the product $K_{2, \{ \}}K_{2, \{1,2\}}$ cannot be written as the sum of all positive $K_{4, \Xi}$'s, but it can be written as
\begin{equation}K_{2, \{ \}}K_{2, \{1,2\}}=K_{4, \{1\} }+K_{4, \{2\} }+K_{4, \{3\} }+K_{4, \{4\} }+K_{4, \{1,3\} }+K_{4, \{1,4\} }+K_{4, \{2,4\} }-K_{4, \{ \} }.
\end{equation}
It would be interesting if such a formula could be found, since it could possibly reveal more about the combinatorial structure of the $K_{n,\Lambda}$'s and possible connections to other areas.

\section{Conclusion}

In this paper, we have studied the properties of a family of formal power series $K_{n,\Lambda}$.
Our most important result shows that the span of these series is a $\mathbb Z$-subalgebra of the space $\mathbb{Z} [[ x_0, x_1, x_2, \ldots, x_{\infty}]]$.
This shows that the answer to Question 2.51 in \cite{grinberg2018} and Question 4.6 in \cite{grinberg2020} is indeed that the statements are true.
Additionally, we have derived a formula to directly find the product $K_{1, \Lambda} K_{m, \Omega}$, and preliminary work shows that similar formulas might exist for other products $K_{n, \Lambda} K_{m, \Omega}$ for $n, m >1$.

Furthermore, we believe that the relations between the $K_{n,\Lambda}$'s and the $L_{n,\Lambda}$'s form useful tools that allowed us to find important properties of those functions and, as such, an algorithm to prove they formed algebras.
We believe that our results and the techniques used in our proof are well-suited to solving similar problems and also have potential future applications in other areas.  In particular, this problem arose from the study of permutation statistics' shuffle-compatibility, and the formula we provide for $K_{1, \Lambda}K_{m, \Omega}$ bears some semblance to the Pieri rule for Schur functions, so it would be interesting to further explore such connections to our work or possible conditions that produce other subalgebras of the ring of formal power series.

\section{Acknowledgements}
The authors would like to thank Professor Darij Grinberg for proposing the research question and his suggestions, advice, and explanations.
The authors would also like to thank the MIT PRIMES-USA program for making this project possible.

\appendix
\section{Proof of formula for \texorpdfstring{$K_{1, \Lambda}K_{m, \Omega}$}{}}\label{apdx}
We prove that for any nonnegative integer $n$ and $\Lambda \in [n]$, 
\begin{equation}
   K_{1, \{ \}}K_{n, \Lambda}=\sum\limits_{s \in S(1, \{ \}, n, \Lambda)} K_{n+1, \text{Gp}(s)}. 
\end{equation}
As such, we prove the equivalent Theorem \ref{generalizedpeakshuffletheorem}, with $n$ and $\Lambda$ replaced by $m$ and $\Omega$ respectively.

\begin{proof}[Proof of Theorem \ref{generalizedpeakshuffletheorem}]
Consider any monomial of the form $\mathbf{x}_g$, where we let $\mathbf{x}_g$ denote the product $x_{g_1}x_{g_2}\cdots x_{g_{n+1}}$ for some $(g_1,g_2,\ldots,g_{n+1}) \in  \mathcal{N}^{n+1}$.  We will show that the coefficient of this monomial is the same for both $\sum\limits_{s \in S(1, \{ \}, n, \Lambda)} K_{n+1, \text{Gp}(s)}$.  For the rest of this proof, let $$M=2^{|\{g_1, g_2, \ldots, g_{n+1}\} \cap \mathbb  N_+|}$$ and let $\mathbf{x}_{g_i}$ denote the exponent of $x_{g_i}$ in $\mathbf{x}_g$; for example, if $\mathbf{x}_g$ is the monomial $x_0^2x_2x_5^3$, then $\mathbf{x}_{g_0}=2$, $\mathbf{x}_{g_1}=0$, and $\mathbf{x}_{g_5}=3$. 

We first find the coefficient of $\mathbf{x}_g$ in $K_{1, \{ \}} K_{n, \Lambda}$, which is equal to $\left(x_0+\sum\limits_{i \in \mathbb  N_+ }2x_i+x_{\infty}\right)(K_{n, \Lambda}).$  We will show that this coefficient is equal to $y_0+\sum\limits_{i \in \mathbb  N_+} y_i + y_{\infty}$, where 
$$y_i=\begin{cases}
    0, & \text{if } \mathbf{x}_{g_i}=0 \text{ or } \frac{\mathbf{x}_g}{x_i} \text{ is not a monomial term in } K_{n, \Lambda} \\
    M, & \text{otherwise}
\end{cases}$$
for $i \in \{ 0, \infty \}$, and 
$$y_i=\begin{cases}
    0, & \text{if } \mathbf{x}_{g_i}=0 \text{ or } \frac{\mathbf{x}_g}{x_i} \text{ is not a monomial term in } K_{n, \Lambda}\\
    M, & \text{if } \mathbf{x}_{g_i}=1 \text{ and } \frac{\mathbf{x}_g}{x_i} \text{ is a monomial term in } K_{n, \Lambda}\\
    2M, & \text{otherwise}
\end{cases}$$
for $i \in \mathbb  N_+$.

Consider each monomial term in the expansion of $K_{n, \Lambda}$; by the definition of $K_{n, \Lambda}$, all monomials in this expansion are of degree $n$.  The product of such a monomial with any term from $K_{1, \{ \}}$, all of which are single degree, can be of the form $\mathbf{x}_g$ only if the monomial is of the form $\frac{\mathbf{x}_g}{x_i}$ for some $i \in \mathcal{N}$ such that $x_i$ has a positive exponent in $\mathbf{x}_g$; all such $x_i$ exist in $K_{1, \{ \}}$.
We will split our proof into cases based on the value of $i$, the terms of $K_{n, \Lambda}$, and the value of $\mathbf{x}_{g_i}$.

\begin{itemize}
    \item Case 1: $\mathbf{x}_{g_i}=0$.  In this case, it is not possible for $x_i$ and a monomial term of $K_{n, \Lambda}$, which specifically must be $\frac{\mathbf{x}_g}{x_i}$, to multiply to $\mathbf{x}_g$; if it were possible, $\frac{\mathbf{x}_g}{x_i}$ being a polynomial would imply that the exponent of $x_i$ in $\mathbf{x}_g$ is at least 1, contradicting $\mathbf{x}_{g_i}=0$.  Therefore, we add 0 to our total coefficient of $\mathbf{x}_g$ for all $i$ such that $\mathbf{x}_{g_i}=0$, in accordance with our case of $y_i=0$.  Then for the following cases, assume $\mathbf{x}_{g_i}>0$.
    \item Case 2: $\frac{\mathbf{x}_g}{x_i} \text{ is not a monomial term in } K_{n, \Lambda}$.  In this case, $x_i$ cannot multiply with any monomial in $K_{n, \Lambda}$ to equal $\mathbf{x}_g$, so we add $0$ to our total coefficient of $\mathbf{x}_g$ in this case as well.  Then for the following cases, assume that $\frac{\mathbf{x}_g}{x_i} \text{ is a monomial term in } K_{n, \Lambda}$.
    \item Case 3: $i \in \{0, \infty \}$.  The coefficient of $x_i$ in $K_{1, \{ \}}$ is 1 and the coefficient of $\frac{\mathbf{x}_g}{x_i}$ in $K_{n, \Lambda}$ is $M$.  Therefore, for both possible values of $i$, we add $M$ to the total coefficient of $\mathbf{x}_g$.
    \item Case 4: $i \in \mathbb  N_+$.  Here, we split into two subcases:
    \begin{itemize}
        \item Subcase 1: $\mathbf{x}_{g_i}=1$.  In this subcase, the coefficient of $x_i$ in $K_{1, \{ \}}$ is 2 and the coefficient of $\frac{\mathbf{x}_g}{x_i}$ in $K_{n, \Lambda}$ is $\frac{M}{2}$.  Therefore, we add $2 \frac{M}{2}=M$ to the total coefficient of $\mathbf{x}_g$.
        \item Subcase 2: $\mathbf{x}_{g_i}>1$.  In this subcase, the coefficient of $x_i$ in $K_{1, \{ \}}$ is 2 and the coefficient of $\mathbf{x}_g$ in $K_{n, \Lambda}$ is $M$.  Therefore, this subcase adds $2M$ to the total coefficient of $\mathbf{x}_g$.
    \end{itemize}
\end{itemize}
Note that all of these values are in accordance with our definitions of the various $y_i$, so it remains to show that the coefficient of $\mathbf{x}_g$ in $\sum\limits_{s \in S(1, \{ \}, n, \Lambda)} K_{n+1, \text{Gp}(s)}$ is also the same value $y_0+\sum\limits_{i \in \mathbb  N_+} y_i + y_{\infty}$.  By the definition of $K_{n, \Lambda}$, the subscripts ($g_1, g_2, \ldots, g_{n+1}) \in \mathcal{N}^{n+1}$ of the monomials $x_{g_1}x_{g_2}\cdots x_{g_{n+1}}$ of our power series $K_{n+1, \{ \text{Gp}(s)\} }$ that sum to $\sum\limits_{s \in S(1, \{ \}, n, \Lambda)} K_{n+1, \text{Gp}(s)}$ must be nondecreasing, so there is exactly one way to assign these $g_i$ to match up with the monomial $\mathbf{x}_g$; precisely, the $g_i$ must be in nondecreasing order.

With this fixed ordering of our $g_i$ in $\mathbf{x}_g$ under consideration, we will find the coefficient of $\mathbf{x}_g$ in each $K_{n+1, \text{Gp}(s)}$ by doing casework on the value of $g_k$ in $\mathbf{x}_g$, where $k$ is the position of $B$ in the shuffle $s$, as well as the value of $\mathbf{x}_{g_k}$.  Note that it is impossible for $\mathbf{x}_{g_k}$ to equal 0 because $g_k$ is guaranteed to be in $\mathbf{x}_g$, so we only consider $\mathbf{x}_{g_k} \geq 1$. 
\begin{itemize}
\item Case 1: $\frac{\mathbf{x}_g}{x_k} \text{ is not a monomial term in } K_{n, \Lambda}$.  This implies that there exists some $z \in \Lambda$ such that $g_{z-1}=g_z=g_{z+1}$.  We will show that in this case, $\mathbf{x}_g$ is not a monomial term in $K_{n+1, \text{Gp}(s)}$ by splitting into two subcases based on whether at least one of $g_{k-1}$ and $g_{k+1}$ is not equal to $g_k$.
    \begin{itemize}
        \item Subcase 1: $g_{k-1}=g_k=g_{k+1}$.  In this case, since $B$ is at position $k$ of the shuffle $s$ and $B$ is higher than $C$ and $D$ in the ordering $ABCD$ (recall that we previously defined the ordering $ABCD$ as $A>B>C>D$), by Definition \ref{GPS} $k$ must be in $\text{Gp} (s)$.  Then, the equality $g_{k-1}=g_k=g_{k+1}$ is not possible for any monomial term of $K_{n+1, \text{Gp}(s)}$, which means $\mathbf{x}_g$ is not a monomial term of $K_{n+1, \text{Gp}(s)}$.
        \item Subcase 2: $g_{k+c} \neq g_k$ for some $c \in \{ -1, 1\}$.  This means $g_{k-c} \neq g_{k+c}$, and thus $g_{k-1} \neq g_{k+1}$, since the $g_i$ are in nondecreasing order.
        Assume for the sake of contradiction that $\mathbf{x}_g$ is a monomial term of $K_{n+1, \text{Gp}(s)}$; this implies that there exists no $p \in \text{Gp}(s)$ such that $g_{p-1}=g_p=g_{p+1}$.  Consider the string form $s_{\Lambda}$ of $\Lambda$ compared to the string $s$; note that removing the $B$ from the $k$th position of $s$ leaves exactly $s_{\Lambda}$ by the definition of a shuffle.  Thus, $\text{Gp}(s_{\Lambda})$ must be the same as $\text{Gp}(s)$ but with the $k$ removed, the positions of any $C$s adjacent to the $B$ in $s$ added (because of the definition of a generalized peak set), and all positions past $k$ lowered by 1.  Note that this removal cannot result in any element $e \in \Lambda$ satisfying $g_{e-1}=g_e=g_{e+1}$ because the only positions that could be added to $\Lambda$ are those adjacent to the removed $B$ at position $k$, but we have already shown that $g_{k-1} \neq g_{k+1}$ in $K_{n+1, \text{Gp}(s)}$.  This contradicts the fact that there exists some $z \in \Lambda$ such that $g_{z-1}=g_z=g_{z+1}$, so  $\mathbf{x}_g$ cannot be a monomial term of $K_{n+1, \text{Gp}(s)}$.
\end{itemize}
    Therefore, we add 0 to our total coefficient of $\mathbf{x}_g$ in this case as well. Then for the following cases, assume that $\frac{\mathbf{x}_g}{x_i} \text{ is a monomial term in } K_{n, \Lambda}$.
    \item Case 2: $g_k=0$ (or $g_k=\infty$).  Then, $g_{k-1}$ (or $g_{k+1}$) must be 0 (or $\infty$) as well, so $g_{k+1}$ (or $g_{k-1}$) cannot also be 0 (or $\infty$) or that would violate the definition of $K_{n+1, \text{Gp}(s)}$.  These possibilities are only possible for one value of $k$; namely, the unique $k$ such that $g_{k+1} \neq 0$ (or $g_{k-1} \neq \infty)$, so this case will add one $M$ to the total coefficient of $\mathbf{x}_g$. 
    \item Case 3: $g_k \in \mathbb  N_+$.  Here, we split into two subcases:
    \begin{itemize}
        \item Subcase 1: $\mathbf{x}_{g_k}=1$.  Then, $g_k$ is unique in $g_1, g_2, \ldots, g_{n+1}$, so there is only one $k$ corresponding to this value of $g_k$, for which we add one $M$ to our total coefficient of $\mathbf{x}_g$.
        \item Subcase 2: $\mathbf{x}_{g_k}>1$.  Then, either $x_{g_{k+1}}$ or $x_{g_{k-1}}$ must be equal to $x_{g_k}$, so if the other is as well then that violates the definition of $K_{n+1, \text{Gp}(s)}$.  There is exactly one value of $k$ for which $x_{g_k}=x_{g_{k+1}}$ and $x_{g_k}\neq x_{g_{k-1}}$ and exactly one value of $k$ for which $x_{g_k}=x_{g_{k-1}}$ and $x_{g_k}\neq x_{g_{k+1}}$.  This gives us two values of $k$ for which $K_{n+1, \text{Gp}(s)}$ contains $\mathbf{x}_g$ as a monomial term, so we add $2M$ to the total coefficient of $\mathbf{x}_g$.
    \end{itemize}
\end{itemize}
Thus, the values we add for $\sum\limits_{s \in S(1, \{ \}, n, \Lambda)} K_{n+1, \text{Gp}(s)}$ correspond exactly to the values of $y_i$ for $K_{1, \{ \}} K_{n, \Lambda}$.  This means that the coefficients of $\mathbf{x}_g$ must be equal for both sides of our equality, and since this is true for all degree $n+1$ monomials $\mathbf{x}_g$, our proof is complete.
\end{proof}
\end{document}